\documentclass[11pt]{article}

\usepackage{amsmath,amssymb,amsfonts,amsthm,mathtools}
\usepackage{enumitem}
\usepackage{microtype}
\usepackage[hidelinks]{hyperref}
\usepackage[nameinlink,capitalize]{cleveref}

\usepackage{aliascnt}

\allowdisplaybreaks
\newtheorem{theorem}{Theorem}[section]

\newaliascnt{proposition}{theorem}
\newtheorem{proposition}[proposition]{Proposition}
\aliascntresetthe{proposition}

\newaliascnt{lemma}{theorem}
\newtheorem{lemma}[lemma]{Lemma}
\aliascntresetthe{lemma}

\newaliascnt{corollary}{theorem}
\newtheorem{corollary}[corollary]{Corollary}
\aliascntresetthe{corollary}

\newaliascnt{algorithm}{theorem}
\newtheorem{algorithm}[algorithm]{Algorithm}
\aliascntresetthe{algorithm}

\theoremstyle{definition}

\newaliascnt{definition}{theorem}
\newtheorem{definition}[definition]{Definition}
\aliascntresetthe{definition}

\newaliascnt{example}{theorem}
\newtheorem{example}[example]{Example}
\aliascntresetthe{example}

\theoremstyle{remark}

\newaliascnt{remark}{theorem}
\newtheorem{remark}[remark]{Remark}
\aliascntresetthe{remark}

\crefname{theorem}{theorem}{theorems}
\Crefname{theorem}{Theorem}{Theorems}

\crefname{proposition}{proposition}{propositions}
\Crefname{proposition}{Proposition}{Propositions}

\crefname{lemma}{lemma}{lemmas}
\Crefname{lemma}{Lemma}{Lemmas}

\crefname{corollary}{corollary}{corollaries}
\Crefname{corollary}{Corollary}{Corollaries}

\crefname{algorithm}{algorithm}{algorithms}
\Crefname{algorithm}{Algorithm}{Algorithms}

\crefname{definition}{definition}{definitions}
\Crefname{definition}{Definition}{Definitions}

\crefname{example}{example}{examples}
\Crefname{example}{Example}{Examples}

\crefname{remark}{remark}{remarks}
\Crefname{remark}{Remark}{Remarks}

\newcommand{\CC}{\mathbb{C}}
\newcommand{\ZZ}{\mathbb{Z}}
\newcommand{\NN}{\mathbb{N}}
\newcommand{\supp}{\operatorname{supp}}
\newcommand{\nsupp}{\operatorname{nsupp}}
\newcommand{\inw}{\operatorname{in}_{w}}
\newcommand{\Hilb}{\operatorname{Hilb}}
\newcommand{\link}{\operatorname{link}}
\newcommand{\frank}{\operatorname{frank}}
\newcommand{\fdeg}{\operatorname{fdeg}}

\newcommand{\rank}{\operatorname{rank}}
\newcommand{\mfrak}{\mathfrak{m}}

\usepackage{bm}

\numberwithin{equation}{section}

\title{Hilbert Series and Logarithmic Degrees of
$A$-Hypergeometric Series}

\author{Mao Nagamine
\thanks{This work was supported by JST SPRING, Grant Number JPMJSP2119.}}
\date{}

\begin{document}
\maketitle
\begin{abstract}
Fix a generic weight vector and a fake exponent of a homogeneous
$A$-hypergeometric system. Using all corresponding standard pairs,
including embedded ones, we construct an Artinian quotient of the
Stanley--Reisner ring of the link of the negative support. Its Hilbert
series gives the graded dimensions of the orthogonal complement of the
local fake indicial ideal and, under the Okuyama--Saito Frobenius
condition, those of the leading logarithmic coefficient space of actual
series solutions. The construction requires no Cohen--Macaulay
hypothesis. When a top-dimensional standard pair occurs and the link is
Cohen--Macaulay, the Hilbert series specializes to the $h$-polynomial
of the link.\par

\end{abstract}

\section{Introduction}

The theory of $A$-hypergeometric systems, initiated by Gel'fand, Kapranov,
and Zelevinsky, relates systems of partial differential equations to toric
ideals, regular triangulations, and combinatorial commutative algebra
\cite{GKZbook,SST}.  Let
$$
A=(a_1,\ldots,a_n)\in\ZZ^{d\times n}
$$
be of rank $d$.  We assume throughout that the columns of $A$ lie in an
affine hyperplane of $\mathbb{Q}^d$ not containing the origin.  For a
parameter $\beta\in\CC^d$, the associated $A$-hypergeometric system is
generated by the toric operators and the Euler operators.  For generic
parameters, its local series solutions are logarithm-free, whereas resonant
parameters may produce logarithmic solutions. Saito obtained a dimension formula for the space of logarithm-free
canonical series solutions in terms of regular triangulations
\cite{Saito2002}.

Saito developed a Frobenius method in which a fake exponent is
perturbed in lattice directions. Okuyama and Saito subsequently refined
this construction: under an explicit ideal-theoretic condition, all series
solutions with a fixed exponent are obtained by applying constant-coefficient
differential operators to a perturbed series
\cite{SaitoLog,OkuyamaSaitoII}.  The
resulting logarithmic coefficients are polynomials in linear forms
$\log x^b$, $b\in\ker_{\ZZ}(A)$.

Saito, Sturmfels, and Takayama used Hilbert-series methods and the
fake indicial ideal to study the holonomic rank. In particular, when
$$
\mathbb{C}[\partial]/\operatorname{in}_{w}(I_A)
$$
is Cohen--Macaulay, their argument yields
$$
\operatorname{rank}(H_A(\beta))
=
\operatorname{vol}(A)
$$
\cite[Section~4.1 and Theorem~4.1.5]{SST}.
This is a global rank formula under a Cohen--Macaulay hypothesis.
The present paper studies the finer contribution attached to each fixed
fake exponent. We construct an Artinian quotient whose Hilbert series
computes the graded dimensions of the orthogonal complement of the
corresponding local fake indicial ideal. This construction remains valid
when $\mathbb{C}[\operatorname{link}_v]$ is not Cohen--Macaulay and when
the fake exponent is supported only on embedded standard pairs.

The purpose of this paper is to determine the graded dimensions of the
leading logarithmic polynomial space without explicitly constructing all
series. Fix a generic weight vector $w$ and a fake exponent $v$, and write
$$
I_0=\operatorname{nsupp}(v).
$$
The standard pairs of $\operatorname{in}_{w}(I_A)$ corresponding to $v$
determine a simplicial complex $\Delta_w(v)$. Since every facet of this
complex contains $I_0$, we consider the link
$$
\operatorname{link}_v
=
\operatorname{link}_{\Delta_w(v)}(I_0).
$$
All corresponding standard pairs, including embedded ones, are used in
this construction.

The main construction is a graded isomorphism
$$
T/Q_v
\cong
\CC[\link_v]/(\ell_1,\ldots,\ell_{d-|I_0|}),
$$
where $Q_v$ is the homogeneous local fake indicial ideal and the $\ell_i$ are
obtained by eliminating the variables indexed by $I_0$ from the Euler
relations.  The quotient on the right is always Artinian.  Consequently, its
Hilbert series
$$
H_{w,v}(t)
=
\Hilb\left(
\CC[\link_v]/(\ell_1,\ldots,\ell_{d-|I_0|}),t
\right)
$$
is a polynomial and satisfies
$$
H_{w,v}(t)
=
\sum_{q\geq0}\dim_{\CC}(Q_v^{\perp})_q t^q.
$$
This formula holds for every fake exponent, even when it corresponds only to
embedded standard pairs.

If the Okuyama--Saito condition
$$
P=m(s)P_B
$$
holds for a $\ZZ$-basis $B$ of $\ker_{\ZZ}(A)$, then $v$ is an exponent and
$Q_v^{\perp}$ is degree-preservingly identified with the space $C_v$ of
leading logarithmic polynomials of actual series solutions.  Hence
$$
H_{w,v}(t)
=
\sum_{q\geq0}\dim_{\CC}C_v(q)t^q.
$$

The familiar $h$-polynomial appears as a specialization.  If at least one
standard pair corresponding to $v$ is top-dimensional and
$\CC[\link_v]$ is Cohen--Macaulay, then the above linear forms form a regular
sequence and
$$
H_{w,v}(t)=h_{\link_v}(t).
$$
The top-dimensionality hypothesis is essential: for an exponent supported
only on an embedded standard pair, the number $d-|I_0|$ of Euler forms can be
strictly larger than $\dim\CC[\link_v]$.  The Cohen--Macaulay hypothesis is
also essential for the $h$-polynomial identity; we give an explicit
non-Cohen--Macaulay example satisfying the Frobenius condition for which
$H_{w,v}(t)\neq h_{\link_v}(t)$.

This paper is organized as follows.  Section~\ref{sec:preliminaries} recalls
$A$-hypergeometric systems, fake exponents, standard pairs, and the
Okuyama--Saito Frobenius input.  Section~\ref{sec:simplicial} introduces the
simplicial complex attached to a fake exponent and the local homogeneous
ideal.  Section~\ref{sec:srmodel} proves the Stanley--Reisner realization and
the general Artinian Hilbert-series theorem.  Section~\ref{sec:cmcase}
derives the Cohen--Macaulay and shellable specializations.  Section
\ref{sec:examples} gives examples exhibiting the embedded, non-pure,
Cohen--Macaulay, and non-Cohen--Macaulay cases.

\section{$A$-hypergeometric systems and Frobenius ideal}
\label{sec:preliminaries}

\subsection{Systems and fake exponents}

Let $\NN=\{0,1,2,\ldots\}$ and write
$$
\partial^u=\partial_1^{u_1}\cdots\partial_n^{u_n}
$$
for $u\in\NN^n$.  The toric ideal associated with $A$ is
$$
I_A
=
\langle
\partial^u-\partial^{u'}
\mid
Au=Au',\ u,u'\in\NN^n
\rangle
\subseteq
\CC[\partial_1,\ldots,\partial_n].
$$
Let
$$
D
=
\CC\langle x_1,\ldots,x_n,\partial_1,\ldots,\partial_n\rangle
$$
be the $n$-th Weyl algebra and set
$$
\vartheta_j=x_j\partial_j.
$$
For $\beta=(\beta_1,\ldots,\beta_d)^T\in\CC^d$, define
$$
H_A(\beta)
=
D I_A
+
D\left\langle
\sum_{j=1}^n a_{ij}\vartheta_j-\beta_i
\ \middle|\
 i=1,\ldots,d
\right\rangle.
$$
The left ideal $H_A(\beta)$ defines the $A$-hypergeometric system
with parameter $\beta$, and
$$
M_A(\beta)=D/H_A(\beta)
$$
is the associated $A$-hypergeometric $D$-module. Under the homogeneity
assumption on $A$, the $D$-module $M_A(\beta)$ is regular holonomic
\cite{Hotta,SchulzeWalther}.

Set
$$
L=\ker_{\ZZ}(A)=\{u\in\ZZ^n\mid Au=0\}.
$$
Fix a generic weight vector $w\in\mathbb{R}^n$.  The fake indicial ideal is
$$
\operatorname{find}_{w}(H_A(\beta))
=
\left(
D\cdot\operatorname{in}_{w}(I_A)
\cap
\mathbb{C}[\vartheta]
\right)
+
\langle A\vartheta-\beta\rangle.
$$
A zero $v\in\CC^n$ of this ideal is called a fake exponent.  It is called an
exponent in the direction $w$ if there is an $A$-hypergeometric series
$$
\phi=x^v\sum_{u\in L}g_u(\log x)x^u
$$
in the direction $w$ with $g_0\neq0$.

For $z\in\CC^n$, define its negative support by
$$
\nsupp(z)=\{j\in[n]\mid z_j\in\ZZ_{<0}\}.
$$
For a fixed fake exponent $v$ and $u\in L$, write
$$
I_u=\nsupp(v+u),
\quad
I_0=\nsupp(v).
$$

\subsection{Standard pairs}

Let $M\subseteq\CC[\partial]$ be a monomial ideal.  A standard pair of $M$ is
a pair $(a,\sigma)$ with $a\in\NN^n$ and $\sigma\subseteq[n]$ satisfying:
\begin{enumerate}[label=(\roman*)]
\item $a_i=0$ for every $i\in\sigma$;
\item $\partial^a\prod_{j\in\sigma}\partial_j^{b_j}\notin M$ for all
$b_j\in\NN$;
\item for every $\ell\notin\sigma$, some monomial obtained by adjoining a
positive power of $\partial_\ell$ lies in $M$.
\end{enumerate}
We denote the set of standard pairs by $\mathcal{S}(M)$.  A standard pair is
called top-dimensional when $|\sigma|=d$; the remaining standard pairs are
embedded.

The standard-pair description of fake exponents gives the following
criterion \cite[Corollary~3.2.3]{SST}.  A vector $v$ is a fake exponent of
$H_A(\beta)$ with respect to $w$ if and only if $Av=\beta$ and there is
$(a,\sigma)\in\mathcal{S}(\inw(I_A))$ such that
$$
v_j=a_j
\quad
(j\notin\sigma).
$$
For a fixed fake exponent $v$, put
$$
\mathcal{S}_w(v)
=
\left\{
(a,\sigma)\in\mathcal{S}(\inw(I_A))
\ \middle|\
 v_j=a_j\text{ for every }j\notin\sigma
\right\}.
$$
All standard pairs in $\mathcal{S}_w(v)$, including embedded ones, will be
used below.

Let
$$
\mathcal{G}
=
\left\{
\partial^{g_+^{(i)}}-\partial^{g_-^{(i)}}
\ \middle|\
 i=1,\ldots,m
\right\}
$$
be the reduced Gr\"obner basis of $I_A$ with respect to $w$, with
$\partial^{g_+^{(i)}}\in\inw(I_A)$.  Put
$$
g^{(i)}=g_+^{(i)}-g_-^{(i)}
$$
and
$$
G^{(i)}=I_{-g^{(i)}}\setminus I_0.
$$

\subsection{Frobenius ideals and leading logarithmic polynomials}

We recall only the part of the Frobenius construction needed in this paper.
Let $B=\{b^{(1)},\ldots,b^{(h)}\}$ be a $\ZZ$-basis of $L$, where
$h=n-d$, and write
$$
(Bs)_j=\sum_{k=1}^h b_j^{(k)}s_k.
$$
For $F\subseteq[n]$, set
$$
(Bs)^F=\prod_{j\in F}(Bs)_j.
$$
Define the affine semigroup generated by the reduced Gr\"obner basis
vectors by
$$
C(w)=\sum_{i=1}^m\NN g^{(i)}.
$$
Following Saito \cite{SaitoLog} and Okuyama--Saito \cite{OkuyamaSaitoII}, let
$$
NS_w(v)
=
\left\{
I_u
\;\middle|\;
u\in L,\ 
u'\in C(w)
\text{ for every }u'\in L\text{ such that }I_{u'}=I_u
\right\},
$$
let $\operatorname{NS}_w^c(v)$ be its complement among the negative supports
$I_u$, and put
$$
K=\bigcap_{I\in\operatorname{NS}_w(v)}I.
$$
The Frobenius ideal and its Gr\"obner-basis approximation are
$$
P
=
\left\langle
(Bs)^{I\cup J\setminus K}
\ \middle|\
I\in\operatorname{NS}_w(v),\ J\in\operatorname{NS}_w^c(v)
\right\rangle,
$$
$$
m(s)=(Bs)^{I_0\setminus K},
$$
and
$$
P_B
=
\left\langle
(Bs)^{G^{(i)}}
\ \middle|\
 i=1,\ldots,m
\right\rangle.
$$
Okuyama and Saito proved
$$
m(s)P_B\subseteq P\subseteq P_B
$$
and showed that the equality
\begin{equation}
P=m(s)P_B
\label{eq:frobenius-condition}
\end{equation}
is sufficient for the Frobenius construction to produce the full space of
series solutions with exponent $v$ \cite[Proposition 3.11 and Theorem~4.4]{OkuyamaSaitoII}.

For an exponent $v$, let $C_v$ denote the vector space of logarithmic
polynomials occurring as the coefficient of the starting monomial $x^v$ in
series solutions with exponent $v$ in the direction $w$.  Write
$$
C_v(q)=C_v\cap\CC[\log x_1,\ldots,\log x_n]_q.
$$
Thus
$$
C_v=\bigoplus_{q\geq0}C_v(q).
$$
Define
$$
\lambda_B:\CC[\partial_{s_1},\ldots,\partial_{s_h}]
\longrightarrow
\CC[\log x_1,\ldots,\log x_n]
$$
by
$$
\lambda_B(q(\partial_s))=q((\log x)B).
$$
In particular,
$$
\lambda_B(\partial_{s_k})
=
\sum_{j=1}^n b_j^{(k)}\log x_j.
$$

\section{The simplicial complex and the local fake indicial ideal}
\label{sec:simplicial}

\subsection{The complex attached to a fake exponent}

\begin{lemma}
\label{lem:I0-in-support}
For every $(a,\sigma)\in\mathcal{S}_w(v)$, one has $I_0\subseteq\sigma$.
\end{lemma}

\begin{proof}
If $j\notin\sigma$, then $v_j=a_j\in\NN$.  Hence $v_j$ is not a negative
integer, and therefore $j\notin I_0$.
\end{proof}

\begin{definition}
\label{def:delta-link}
Define
$$
\Delta_w(v)
=
\left\{
\tau\subseteq[n]
\ \middle|\
\tau\subseteq\sigma\text{ for some }(a,\sigma)\in\mathcal{S}_w(v)
\right\}.
$$
The link attached to $v$ is
$$
\link_v
=
\link_{\Delta_w(v)}(I_0)
=
\left\{
\tau\in\Delta_w(v)
\ \middle|\
\tau\cap I_0=\emptyset,\ \tau\cup I_0\in\Delta_w(v)
\right\}.
$$
\end{definition}

The facets of $\Delta_w(v)$ are the inclusion-maximal supports $\sigma$
occurring in $\mathcal{S}_w(v)$, and the facets of $\link_v$ are the sets
$\sigma\setminus I_0$ obtained from those maximal supports.

Let
$$
V=[n]\setminus I_0
$$
and define the squarefree monomial ideal
$$
J_w(v)
=
\left\langle
\prod_{j\in G^{(i)}}x_j
\ \middle|\
 i=1,\ldots,m
\right\rangle
\subseteq
\CC[x_j\mid j\in V].
$$

\subsection{Stanley--Reisner and Hilbert-series facts}

For a simplicial complex $\Delta$ on a vertex set $V$, let
$$
I_\Delta
=
\left\langle
\prod_{i\in\tau}x_i
\ \middle|\
\tau\notin\Delta
\right\rangle
$$
and
$$
\CC[\Delta]=\CC[x_i\mid i\in V]/I_\Delta.
$$
If $\dim\Delta=r-1$, its Hilbert series has the form
$$
\Hilb(\CC[\Delta],t)=\frac{h_\Delta(t)}{(1-t)^r},
$$
where $h_\Delta(t)$ is the $h$-polynomial.  Equivalently, if
$f_j(\Delta)$ denotes the number of $j$-dimensional faces and
$f_{-1}(\Delta)=1$, then
$$
h_\Delta(t)
=
\sum_{j=-1}^{r-1}f_j(\Delta)t^{j+1}(1-t)^{r-1-j}.
$$

We use the following standard facts from combinatorial commutative algebra
\cite{MillerSturmfels}.  If $\CC[\Delta]$ is Cohen--Macaulay,
every homogeneous system of parameters is a regular sequence.  If
$\ell_1,\ldots,\ell_r$ is a regular sequence of linear forms on
$\CC[\Delta]$, then
\begin{equation}
\Hilb\left(\CC[\Delta]/(\ell_1,\ldots,\ell_r),t\right)
=h_\Delta(t).
\label{eq:h-artinian}
\end{equation}
A pure shellable complex is Cohen--Macaulay.  We will also use the elementary
one-dimensional criterion below.

\begin{proposition}
\label{prop:graph-cm}
Let $\Delta$ be a one-dimensional simplicial complex.  If $\Delta$ is
connected and has no isolated vertices, then $\CC[\Delta]$ is
Cohen--Macaulay.
\end{proposition}

\begin{proof}
This follows immediately from Reisner's criterion: connectedness gives the
vanishing of reduced homology in degree zero for the empty face, the link of
each vertex is a nonempty zero-dimensional complex, and the link of each edge
is $\{\emptyset\}$.
\end{proof}

\subsection{The differential pairing}

Set
$$
T=\CC[\theta_1,\ldots,\theta_n],
\quad
U=\CC[z_1,\ldots,z_n].
$$
For $\phi\in T$ and $f\in U$, define
$$
\langle\phi,f\rangle
=
\left.\phi(\partial_z)\bullet f(z)\right|_{z=0}.
$$
For a homogeneous ideal $Q\subseteq T$, define
$$
Q^{\perp}
=
\{f\in U\mid \phi(\partial_z)\bullet f=0\text{ for all }\phi\in Q\}.
$$

\begin{proposition}
\label{prop:pairing-hilbert}
For every homogeneous ideal $Q\subseteq T$,
$$
\dim_{\CC}(Q^{\perp})_q
=
\dim_{\CC}(T/Q)_q
$$
for every $q\geq0$.  Consequently,
$$
\sum_{q\geq0}\dim_{\CC}(Q^{\perp})_q t^q
=
\Hilb(T/Q,t).
$$
\end{proposition}

\begin{proof}
The monomial bases of $T_q$ and $U_q$ are dual up to nonzero factorials, hence
the pairing $T_q\times U_q\to\CC$ is perfect.  Since $Q$ is homogeneous,
$(Q^{\perp})_q$ is the orthogonal complement of $Q\cap T_q$ in $U_q$.
Therefore
$$
\dim_{\CC}(Q^{\perp})_q
=
\dim_{\CC}T_q-\dim_{\CC}(Q\cap T_q)
=
\dim_{\CC}(T/Q)_q.
$$
\end{proof}

To describe the local fake indicial ideal at $v$, distinguish the original
Euler variables $\vartheta_j$ from the centered variables
$$
\theta_j=\vartheta_j-v_j.
$$
Define
$$
M_\theta(v)
=
\left\langle
\prod_{j\in G^{(i)}}\theta_j
\ \middle|\
 i=1,\ldots,m
\right\rangle
$$
and
$$
Q_v=\langle A\theta\rangle+M_\theta(v)\subseteq T.
$$
The space $Q_v^{\perp}$ is the space of homogeneous logarithmic polynomials
satisfying the local fake indicial equation at $v$.

Let
$$
\Phi_B:T\longrightarrow\CC[s_1,\ldots,s_h]
$$
be given by $\Phi_B(\theta_j)=(Bs)_j$.  Then $P_B=\Phi_B(Q_v)$, and because
$B$ is a basis of $L$, $\Phi_B$ induces a degree-preserving isomorphism
$$
T/Q_v\cong\CC[s]/P_B
$$
and a dual degree-preserving isomorphism between $Q_v^{\perp}$ and
$P_B^{\perp}$ \cite[Proposition~3.4 and Theorem~3.14]{OkuyamaSaitoII}.

\begin{theorem}
\label{thm:frobenius-leading-space}
Assume that $B$ is a $\ZZ$-basis of $L$ and that the Frobenius condition
\eqref{eq:frobenius-condition} holds.  Then $v$ is an exponent and
$$
C_v=\lambda_B(P_B^{\perp})\cong Q_v^{\perp}
$$
as graded vector spaces.  In particular,
$$
\dim_{\CC}C_v(q)=\dim_{\CC}(Q_v^{\perp})_q
$$
for all $q\geq0$.
\end{theorem}

\begin{proof}
\cite{OkuyamaSaitoII} gives a spanning family of all series solutions
with exponent $v$ from operators in $P^{\perp}$.  Under
$P=m(s)P_B$, the operators occurring in the coefficient of the starting
monomial identify with $P_B^{\perp}$.  The map $\lambda_B$ converts these
operators into logarithmic polynomials.  The final graded isomorphism follows
from the duality induced by $\Phi_B$.
\end{proof}

\section{The Artinian Stanley--Reisner model}
\label{sec:srmodel}

\subsection{The local monomial ideal}

\begin{lemma}
\label{lem:local-standard-pair-decomposition}
Let $M=\inw(I_A)$.  Then
$$
M_\theta(v)
=
\bigcap_{(a,\sigma)\in\mathcal{S}_w(v)}
\langle\theta_j\mid j\notin\sigma\rangle.
$$
\end{lemma}

\begin{proof}
Let $R=\CC[\vartheta_1,\ldots,\vartheta_n]$ and let
$$
\widetilde{M}=D\cdot M\cap R
$$
be the distraction of $M$.  Put
$$
\mfrak_v
=
\langle\vartheta_1-v_1,\ldots,\vartheta_n-v_n\rangle
=
\langle\theta_1,\ldots,\theta_n\rangle.
$$
Since the initial monomials in the reduced Gr\"obner basis minimally generate
$M$, the distraction is generated by
$$
[\vartheta]_{g_+^{(i)}}
=
\prod_{j=1}^n\prod_{\nu=0}^{g_{+,j}^{(i)}-1}(\vartheta_j-\nu),
\quad i=1,\ldots,m.
$$
In $R_{\mfrak_v}$, a factor $\vartheta_j-\nu$ is a unit unless $\nu=v_j$.
The nonunit factors at $v$ are precisely $\vartheta_j-v_j=\theta_j$ with
$j\in G^{(i)}$.  Hence
$$
\widetilde{M}R_{\mfrak_v}=M_\theta(v)R_{\mfrak_v}.
$$

On the other hand, the standard-pair decomposition of the distraction \cite[Theorem~3.2.2 and Corollary~3.2.3]{SST} gives
$$
\widetilde{M}
=
\bigcap_{(a,\sigma)\in\mathcal{S}(M)}
\langle\vartheta_j-a_j\mid j\notin\sigma\rangle.
$$
After localization at $\mfrak_v$, a component becomes the unit ideal unless
$v_j=a_j$ for every $j\notin\sigma$.  Thus only the pairs in
$\mathcal{S}_w(v)$ remain, and their localized components are
$$
\langle\vartheta_j-v_j\mid j\notin\sigma\rangle
=
\langle\theta_j\mid j\notin\sigma\rangle.
$$
Both ideals are homogeneous in the centered variables, and a homogeneous
ideal is recovered from its localization at the homogeneous maximal ideal.
\end{proof}

\begin{theorem}
\label{thm:link-main-properties}
For every fake exponent $v$, the following statements hold.
\begin{enumerate}[label=(\roman*)]
\item The columns of $A$ indexed by $I_0$ are linearly independent.
\item Under $\theta_j\mapsto x_j$ for $j\in V$,
$$
J_w(v)=I_{\link_v}.
$$
Consequently, the inclusion-minimal sets among
$G^{(1)},\ldots,G^{(m)}$ are precisely the minimal nonfaces of $\link_v$.
\item
$$
\max_{\tau\in\link_v}|\tau|
=
\max_{(a,\sigma)\in\mathcal{S}_w(v)}(|\sigma|-|I_0|)
\leq d-|I_0|.
$$
\item If $\mathcal{S}_w(v)$ contains a top-dimensional standard pair, then
$$
\dim\link_v+1=d-|I_0|.
$$
\end{enumerate}
\end{theorem}

\begin{proof}
We first prove that the columns indexed by every support $\sigma$ occurring
in $\mathcal{S}_w(v)$ are linearly independent.  If not, there is a nonzero
$u\in\ker_{\ZZ}(A)$ with $\supp(u)\subseteq\sigma$.  Write
$u=u_+-u_-$ with disjoint supports.  Since $w$ is generic, after replacing
$u$ by $-u$ if necessary, we may assume $w\cdot u_+>w\cdot u_-$.  Then $\partial^{u_+}\in\operatorname{in}_{w}(I_A)$.
Since $\operatorname{supp}(u_+)\subseteq\sigma$, the defining property of
the standard pair gives
$$
\partial^a\partial^{u_+}\notin\inw(I_A),
$$
a contradiction.  Thus $A_\sigma$ has independent columns.  Since
$I_0\subseteq\sigma$ by \Cref{lem:I0-in-support}, this proves (i).

By \Cref{lem:local-standard-pair-decomposition},
$$
M_\theta(v)
=
\bigcap_{(a,\sigma)\in\mathcal{S}_w(v)}
\langle\theta_j\mid j\notin\sigma\rangle.
$$
Because $I_0\subseteq\sigma$, the condition $j\notin\sigma$ is equivalent to
$j\in V\setminus(\sigma\setminus I_0)$.  The link is the downward closure of the sets $\sigma\setminus I_0$.
Consequently, the displayed decomposition is exactly the
Stanley--Reisner prime decomposition of $I_{\operatorname{link}_v}$.  This proves (ii).

Every face $\tau$ of the link is contained in some
$\sigma\setminus I_0$, and each $\sigma\setminus I_0$ is itself a face.
Therefore
$$
\max_{\tau\in\link_v}|\tau|
=
\max_{(a,\sigma)\in\mathcal{S}_w(v)}(|\sigma|-|I_0|).
$$
The first paragraph gives $|\sigma|\leq d$, proving (iii).  If a
top-dimensional pair occurs, then some $|\sigma|=d$, and equality follows.
\end{proof}

\subsection{Elimination of the negative-support variables}

Write
$$
k=|I_0|,
\quad
I_0=\{i_1,\ldots,i_k\}.
$$

\begin{lemma}
\label{lem:euler-elimination}
The generators of $\langle A\theta\rangle$ can be chosen in the form
$$
\theta_{i_1}-L_1(\theta_V),\ldots,
\theta_{i_k}-L_k(\theta_V),
\ell_1(\theta_V),\ldots,
\ell_{d-k}(\theta_V),
$$
where $\theta_V=(\theta_j)_{j\in V}$ and every $L_p$ and $\ell_q$ is a
linear form in the variables indexed by $V$.
\end{lemma}

\begin{proof}
By \Cref{thm:link-main-properties}(i), $\rank A_{I_0}=k$.  After reordering
columns, suppose $I_0=\{1,\ldots,k\}$.  There is
$U\in\operatorname{GL}_d(\CC)$ such that
$$
UA_{I_0}=\begin{pmatrix}I_k\\0\end{pmatrix}.
$$
The first $k$ rows of $UA\theta$ eliminate the variables indexed by $I_0$,
and the remaining rows involve only $\theta_V$.
\end{proof}

\begin{proposition}
\label{prop:stanley-reisner-realization}
Under $\theta_j\mapsto x_j$ for $j\in V$, the linear forms from
\Cref{lem:euler-elimination} define elements
$$
\ell_1,\ldots,\ell_{d-|I_0|}\in\CC[\link_v]
$$
and there is a degree-preserving isomorphism
\begin{equation}
T/Q_v
\cong
\CC[\link_v]/(\ell_1,\ldots,\ell_{d-|I_0|}).
\label{eq:sr-realization}
\end{equation}
If a top-dimensional standard pair occurs, then
$$
d-|I_0|=\dim\CC[\link_v]=\dim\link_v+1.
$$
\end{proposition}

\begin{proof}
The relations $\theta_{i_p}=L_p(\theta_V)$ eliminate the variables indexed
by $I_0$.  Since every $G^{(i)}$ is contained in $V$, the remaining monomial
relations are $M_\theta(v)$.  By \Cref{thm:link-main-properties}(ii),
$$
\CC[\theta_j\mid j\in V]/M_\theta(v)
\cong
\CC[\link_v].
$$
The remaining Euler relations are the $\ell_i$, which proves
\eqref{eq:sr-realization}.  The dimension statement follows from
\Cref{thm:link-main-properties}(iv).
\end{proof}

\begin{lemma}
\label{lem:artinian-quotient}
For every fake exponent $v$, the quotient in \eqref{eq:sr-realization} is a
finite-dimensional $\CC$-vector space.
\end{lemma}

\begin{proof}
After reordering columns, write
$$
UA
=
\begin{pmatrix}
I_k&C\\
0&B'
\end{pmatrix}
$$
as in \Cref{lem:euler-elimination}.  The forms
$\ell_1,\ldots,\ell_{d-k}$ are the rows of $B'\theta_V$.

For a face $\tau\in\link_v$, let $B'_\tau$ be the submatrix formed by the
columns indexed by $\tau$.  We claim that
$$
\rank B'_\tau=|\tau|.
$$
Choose $(a,\sigma)\in\mathcal{S}_w(v)$ with
$I_0\cup\tau\subseteq\sigma$.  The proof of
\Cref{thm:link-main-properties}(i) shows that the columns of
$A_{I_0\cup\tau}$ are linearly independent.  If $B'_{\tau}c=0$, then
$$
UA_{\tau}c=(C_{\tau}c,0)^T
$$
lies in the span of $UA_{I_0}$.
The linear independence of the columns of $A_{I_0\cup\tau}$ therefore
forces $c=0$.

Now let
$$
z\in V\bigl(I_{\link_v}+(\ell_1,\ldots,\ell_{d-k})\bigr)
$$
and put $\tau=\supp(z)$.  Since $z\in V(I_{\link_v})$, the support $\tau$
is a face.  The equations $\ell_i(z)=0$ give $B'_\tau z_\tau=0$, and the
claim implies $z=0$.  Hence
$$
V\bigl(I_{\link_v}+(\ell_1,\ldots,\ell_{d-k})\bigr)=\{0\}.
$$
Hilbert's Nullstellensatz gives
$$
\sqrt{I_{\link_v}+(\ell_1,\ldots,\ell_{d-k})}
=
\langle x_j\mid j\in V\rangle.
$$
Thus a power of the homogeneous maximal ideal lies in the defining ideal.
It follows that the quotient is finite-dimensional.
\end{proof}

\subsection{The general Artinian Hilbert series}

\begin{definition}
\label{def:artinian-hilbert-series}
For a fake exponent $v$, define
$$
\mathcal{A}_{w,v}
=
\CC[\link_v]/(\ell_1,\ldots,\ell_{d-|I_0|})
$$
and
$$
H_{w,v}(t)=\Hilb(\mathcal{A}_{w,v},t).
$$
By \Cref{lem:artinian-quotient}, this Hilbert series is a polynomial.  Define
$$
\fdeg_w(v)=\deg H_{w,v}(t)
$$
and
$$
\frank_w(v)=\dim_{\CC}Q_v^{\perp}.
$$
\end{definition}

The row reduction used to obtain the $\ell_i$ is not unique, but
\eqref{eq:sr-realization} identifies every resulting quotient with the fixed
graded algebra $T/Q_v$.  Hence $H_{w,v}(t)$ is independent of all choices.

\begin{theorem}
\label{thm:general-artinian-hilbert-series}
For every fake exponent $v$,
$$
\sum_{q\geq0}\dim_{\CC}(Q_v^{\perp})_q t^q
=
H_{w,v}(t).
$$
Consequently,
$$
\max\{q\geq0\mid (Q_v^{\perp})_q\neq0\}
=
\fdeg_w(v)
$$
and
$$
\frank_w(v)=H_{w,v}(1).
$$
\end{theorem}

\begin{proof}
By \Cref{prop:stanley-reisner-realization},
$$
\Hilb(T/Q_v,t)=H_{w,v}(t).
$$
Apply \Cref{prop:pairing-hilbert}.
\end{proof}

\begin{corollary}
\label{cor:general-log-degree}
Assume that $B$ is a $\ZZ$-basis of $L$ and that
$P=m(s)P_B$.  Then
$$
\sum_{q\geq0}\dim_{\CC}C_v(q)t^q
=
H_{w,v}(t).
$$
In particular,
$$
\max\{q\geq0\mid C_v(q)\neq0\}=\fdeg_w(v)
$$
and
$$
\dim_{\CC}C_v=H_{w,v}(1).
$$
\end{corollary}

\begin{proof}
Combine \Cref{thm:frobenius-leading-space,thm:general-artinian-hilbert-series}.
\end{proof}

\section{The Cohen--Macaulay specialization}
\label{sec:cmcase}

\begin{theorem}
\label{thm:cm-regular-sequence}
Assume that $\mathcal{S}_w(v)$ contains a top-dimensional standard pair and
that $\CC[\link_v]$ is Cohen--Macaulay.  Then
$$
\ell_1,\ldots,\ell_{d-|I_0|}
$$
form a regular sequence on $\CC[\link_v]$.
\end{theorem}

\begin{proof}
By \Cref{prop:stanley-reisner-realization}, the number of forms is
$\dim\CC[\link_v]$.  By \Cref{lem:artinian-quotient}, their quotient is finite-dimensional. Since the defining
ideal is homogeneous and generated in positive degree, the quotient is
nonzero. Consequently, the forms constitute a homogeneous system of
parameters. Cohen--Macaulayness implies that every homogeneous system
of parameters is a regular sequence.
\end{proof}

\begin{corollary}
\label{cor:h-polynomial-case}
Under the hypotheses of \Cref{thm:cm-regular-sequence},
$$
H_{w,v}(t)=h_{\link_v}(t).
$$
Consequently,
$$
\sum_{q\geq0}\dim_{\CC}(Q_v^{\perp})_q t^q
=h_{\link_v}(t)
$$
and
$$
\frank_w(v)=h_{\link_v}(1).
$$
If, in addition, $B$ is a $\ZZ$-basis of $L$ and $P=m(s)P_B$, then
$$
\sum_{q\geq0}\dim_{\CC}C_v(q)t^q
=h_{\link_v}(t).
$$
Thus
$$
\max\{q\mid C_v(q)\neq0\}=\deg h_{\link_v}(t),
\quad
\dim_{\CC}C_v=h_{\link_v}(1).
$$
\end{corollary}

\begin{proof}
Since the regular sequence has length $\dim\CC[\link_v]$, \eqref{eq:h-artinian} applies.  The remaining assertions follow from
\Cref{thm:general-artinian-hilbert-series} and \Cref{cor:general-log-degree}.
\end{proof}

\begin{corollary}
\label{cor:shellability}
Assume that $\mathcal{S}_w(v)$ contains a top-dimensional standard pair and
that $\link_v$ is pure shellable.  Then all conclusions of
\Cref{cor:h-polynomial-case} hold.
\end{corollary}

\begin{proof}
A pure shellable simplicial complex is Cohen--Macaulay
\cite{BjornerWachs}.
\end{proof}

\begin{remark}
The top-dimensionality assumption cannot be removed from the
$h$-polynomial formula.  In general,
$$
\dim\link_v+1\leq d-|I_0|,
$$
and the inequality can be strict for a fake exponent corresponding only to
embedded standard pairs.  In that case the full collection of
$d-|I_0|$ Euler forms cannot be a regular sequence when its length exceeds
$\dim\CC[\link_v]$.  The Artinian model and
\Cref{thm:general-artinian-hilbert-series}, however, remain valid.
\end{remark}

\begin{algorithm}[Computing logarithmic degrees]
\label{alg:computation}
For a fake exponent $v$, the graded dimensions can be computed as follows.
\begin{enumerate}[label=\arabic*.]
\item Compute $I_0=\nsupp(v)$ and all pairs in $\mathcal{S}_w(v)$, including
embedded pairs.  Form $\link_v$ by removing $I_0$ from the maximal supports.
\item Compute the sets $G^{(i)}=I_{-g^{(i)}}\setminus I_0$ from the reduced
Gr\"obner basis.  Their inclusion-minimal members are the minimal nonfaces of
$\link_v$.
\item Row-reduce the Euler relations as in
\Cref{lem:euler-elimination}.  This yields the linear forms
$\ell_1,\ldots,\ell_{d-|I_0|}$.
\item Compute
$$
H_{w,v}(t)
=
\Hilb\left(\CC[\link_v]/(\ell_1,\ldots,\ell_{d-|I_0|}),t\right).
$$
This gives the graded dimensions of $Q_v^{\perp}$ for every fake exponent.
\item If a top-dimensional pair occurs and $\CC[\link_v]$ is
Cohen--Macaulay, replace this quotient computation by the
$h$-polynomial $h_{\link_v}(t)$.
\item If $P=m(s)P_B$, the same polynomial gives the graded dimensions of the
actual leading logarithmic polynomial space $C_v$.
\end{enumerate}
\end{algorithm}

\section{Examples}
\label{sec:examples}

\subsection{A non-Cohen--Macaulay link}

\begin{example}
\label{ex:non-cm-link}
Let
$$
A
=
\begin{pmatrix}
1&1&1&1&1\\
0&0&1&1&1\\
1&0&1&0&3
\end{pmatrix},
\quad
w=(2,1,1,2,2).
$$
The reduced Gr\"obner basis is
$$
\left\{
\partial_4^2\partial_5-\partial_3^3,
\ \partial_1^2\partial_3-\partial_2^2\partial_5,
\ \partial_2\partial_4\partial_5-\partial_1\partial_3^2,
\ \partial_1\partial_4-\partial_2\partial_3
\right\},
$$
with initial monomials displayed on the left.  Hence
$$
\inw(I_A)
=
\langle
\partial_4^2\partial_5,
\partial_1^2\partial_3,
\partial_2\partial_4\partial_5,
\partial_1\partial_4
\rangle.
$$
Its standard pairs are
$$
(0,\ast,\ast,\ast,0),
\quad
(\ast,\ast,0,0,\ast),
\quad
(0,\ast,\ast,0,\ast),
\quad
(1,\ast,\ast,0,\ast),
\quad
(0,0,\ast,1,\ast).
$$

Take
$$
v=(0,0,-1,1,0),
\quad
Av=(0,0,-1)^T,
\quad
I_0=\{3\}.
$$
The corresponding standard pairs are
$$
(0,\ast,\ast,\ast,0)
\quad\text{and}\quad
(0,0,\ast,1,\ast),
$$
with supports $\{2,3,4\}$ and $\{3,5\}$.  Thus the facets of the link are
$$
\{2,4\},
\quad
\{5\}.
$$
Since the link consists of one edge and one isolated vertex, it is not
Cohen--Macaulay by Reisner's criterion.

The four Gr\"obner basis elements give
$$
G^{(1)}=\{4,5\},
\quad
G^{(2)}=\{1\},
\quad
G^{(3)}=\{2,5\},
\quad
G^{(4)}=\{1\}.
$$
Hence
$$
I_{\link_v}
=
\langle x_1,x_2x_5,x_4x_5\rangle
\subseteq\CC[x_1,x_2,x_4,x_5].
$$
The Euler ideal can be written as
$$
\langle A\theta\rangle
=
\langle
\theta_3+\theta_4+\theta_5,
\ell_1,
\ell_2
\rangle,
$$
where
$$
\ell_1=\theta_1+\theta_2,
\quad
\ell_2=\theta_1-\theta_4+2\theta_5.
$$
Therefore
$$
\mathcal{A}_{w,v}
\cong
\frac{\CC[x_1,x_2,x_4,x_5]}
{\langle
x_1,x_2x_5,x_4x_5,x_1+x_2,x_1-x_4+2x_5
\rangle}
\cong
\CC[x_5]/(x_5^2).
$$
Thus
$$
H_{w,v}(t)=1+t.
$$

On the other hand, the $f$-vector of the link is $(1,3,1)$, and
$$
h_{\link_v}(t)
=(1-t)^2+3t(1-t)+t^2
=1+t-t^2.
$$
Hence
$$
H_{w,v}(t)\neq h_{\link_v}(t).
$$
This is the precise point at which the Cohen--Macaulay specialization fails.

The Frobenius condition nevertheless holds.  Take
$$
b^{(1)}=(2,-2,1,0,-1)^T,
\quad
b^{(2)}=(-1,1,-2,1,1)^T.
$$
A direct negative-support computation gives
$$
\operatorname{NS}_w(v)
=
\bigl\{\{3\},\{2,3\},\{2,3,5\}\bigr\},
\quad
K=I_0=\{3\}.
$$
Thus $m(s)=1$ and $P=P_B$.  Moreover,
$$
P_B
=
\langle 2s_1-s_2,\ (-s_1+2s_2)s_2\rangle
=
\langle 2s_1-s_2,s_2^2\rangle.
$$
Consequently,
$$
P_B^{\perp}
=
\CC\{1,\partial_{s_1}+2\partial_{s_2}\}.
$$
Since
$$
\lambda_B(\partial_{s_1}+2\partial_{s_2})
=
-3\log x_3+2\log x_4+\log x_5,
$$
we obtain
$$
C_v
=
\CC\{1,-3\log x_3+2\log x_4+\log x_5\}.
$$
The actual maximal logarithmic degree is $1$ and $\dim_{\CC}C_v=2$.
\end{example}

\subsection{A fake exponent supported only on an embedded pair}

\begin{example}
\label{ex:embedded-only}
Let
$$
A
=
\begin{pmatrix}
1&1&1&1&1&1\\
0&1&1&0&-1&-1\\
-1&-1&0&1&1&0
\end{pmatrix},
\quad
w=(2,3,8,1,13,5).
$$
The initial ideal is
$$
\begin{aligned}
\inw(I_A)=\langle
\partial_2\partial_5,
\partial_3\partial_6,
\partial_3\partial_5^2,
\partial_3^2\partial_5,
\partial_1\partial_5^2,
\partial_1\partial_3\partial_5,
\partial_1\partial_3^2,
\partial_1^2\partial_5,
\partial_2^2\partial_6^2,
\partial_1^2\partial_3
\rangle.
\end{aligned}
$$
Take
$$
v=(1,0,0,0,1,-2),
\quad
Av=(0,1,0)^T,
\quad
I_0=\{6\}.
$$
The only corresponding standard pair is
$$
(1,0,0,\ast,1,\ast),
$$
with support $\sigma=\{4,6\}$.  It is embedded, and there is no
corresponding top-dimensional pair.  Therefore
$$
\link_v=\{\emptyset,\{4\}\}
$$
and
$$
\max_{\tau\in\link_v}|\tau|=1<2=d-|I_0|.
$$

The sets $G^{(i)}$ obtained from the ten reduced Gr\"obner basis elements are
$$
\{2\},\{3\},\{3,5\},\{3\},\{5\},\{3\},\{3\},\{1\},\{2\},\{1,3\}.
$$
Their inclusion-minimal members are 
$$\{1\},\{2\},\{3\},\{5\}.
$$
It follows that
$$
\CC[\link_v]\cong\CC[x_4].
$$
The Euler ideal can be written as
$$
\langle A\theta\rangle
=
\langle
\theta_6-(\theta_2+\theta_3-\theta_5),
\ell_1,
\ell_2
\rangle,
$$
where
$$
\ell_1=\theta_1+2\theta_2+2\theta_3+\theta_4,
\quad
\ell_2=-\theta_1-\theta_2+\theta_4+\theta_5.
$$
Both forms map to $x_4$ in $\CC[\link_v]$.  Hence
$$
\mathcal{A}_{w,v}
\cong
\CC[x_4]/(x_4,x_4)
\cong
\CC
$$
and
$$
H_{w,v}(t)=1.
$$
Thus
$$
\frank_w(v)=1,
\quad
\fdeg_w(v)=0.
$$
The two Euler forms do not form a regular sequence: after quotienting by the
first, the image of the second is zero.  This example shows why the general
Artinian model is needed beyond the top-dimensional case.
\end{example}

\subsection{A one-dimensional Cohen--Macaulay link}

\begin{example}
\label{ex:path-link}
Continue with the matrix and weight of \Cref{ex:embedded-only}. Then
$$
v=(0,0,0,-1,0,0)
$$
is a fake exponent, and
$$
I_0=\operatorname{nsupp}(v)=\{4\}.
$$
The link has facets
$$
\{1,2\},\quad\{1,6\},\quad\{2,3\},\quad\{5,6\}.
$$
These facets form the path
$$
3-2-1-6-5.
$$
It is connected and has no isolated vertices; hence its Stanley--Reisner ring
is Cohen--Macaulay by \Cref{prop:graph-cm}.  Its Stanley--Reisner ideal is
$$
I_{\link_v}
=
\langle
x_1x_3,x_1x_5,x_2x_5,x_2x_6,x_3x_5,x_3x_6
\rangle.
$$

The Euler ideal can be written as
$$
\langle A\theta\rangle
=
\langle
\theta_4-(\theta_1+\theta_2-\theta_5),
\ell_1,
\ell_2
\rangle,
$$
where
$$
\ell_1=\theta_2+\theta_3-\theta_5-\theta_6,
\quad
\ell_2=2\theta_1+2\theta_2+\theta_3+\theta_6.
$$
The $f$-vector is $(1,5,4)$, and hence
$$
h_{\link_v}(t)
=(1-t)^2+5t(1-t)+4t^2
=1+3t.
$$
Therefore
$$
H_{w,v}(t)=1+3t,
\quad
\frank_w(v)=4,
\quad
\fdeg_w(v)=1.
$$
A direct quotient computation gives a graded basis
$$
1,x_3,x_5,x_6.
$$

In this example $I_0$ is contained in the intersection of the supports of all
standard pairs of $\inw(I_A)$.  The global standard-pair criterion of
\cite{NagamineCore} gives $K=I_0$, and hence $P=P_B$.  Consequently,
$$
\sum_{q\geq0}\dim_{\CC}C_v(q)t^q=1+3t.
$$
Thus the actual maximal logarithmic degree is $1$ and
$\dim_{\CC}C_v=4$.
\end{example}

\subsection{A two-dimensional Cohen--Macaulay link}

\begin{example}
\label{ex:two-dimensional-link}
Let
$$
A
=
\begin{pmatrix}
1&1&1&1&1\\
-1&1&1&-1&0\\
-1&-1&1&1&0
\end{pmatrix},
\quad
\beta=(1,0,0)^T,
\quad
w=(1,1,1,1,0).
$$
Then
$$
v=(0,0,0,0,1)
$$
is the unique exponent and $I_0=\emptyset$.  The reduced Gr\"obner basis is
$$
\partial_1\partial_3-\partial_5^2,
\quad
\partial_2\partial_4-\partial_5^2,
$$
with initial monomials displayed on the left.  Hence
$$
I_{\link_v}=\langle x_1x_3,x_2x_4\rangle.
$$
The facets are
$$
\{1,2,5\},\quad\{1,4,5\},\quad\{2,3,5\},\quad\{3,4,5\}.
$$
Thus the link is the cone with apex $5$ over the four-cycle
$$
1-2-3-4-1.
$$
It is two-dimensional and Cohen--Macaulay. Its $f$-vector is
$$
(1,5,8,4).
$$
Hence
$$
h_{\operatorname{link}_v}(t)
=
(1-t)^3+5t(1-t)^2+8t^2(1-t)+4t^3
=
1+2t+t^2.
$$

Let
$$
b^{(1)}=(1,0,1,0,-2)^T,
\quad
b^{(2)}=(0,1,0,1,-2)^T.
$$
Then
$$
B=\{b^{(1)},b^{(2)}\}
$$
is a $\mathbb{Z}$-basis of $L$.
Since $K=I_0=\emptyset$, we have $P=P_B$, where
$$
P_B=\langle s_1^2,s_2^2\rangle.
$$
Thus
$$
P_B^{\perp}
=
\CC\{1,\partial_{s_1},\partial_{s_2},\partial_{s_1}\partial_{s_2}\}.
$$
Consequently,
$$
\sum_{q\geq0}\dim_{\CC}C_v(q)t^q=1+2t+t^2.
$$
The actual maximal logarithmic degree is $2$ and $\dim_{\CC}C_v=4$.
Writing
$$
L_1=\log x_1+\log x_3-2\log x_5,
\quad
L_2=\log x_2+\log x_4-2\log x_5,
$$
a graded basis is
$$
1,\quad L_1,\quad L_2,\quad L_1L_2.
$$
\end{example}

\section*{Use of AI}

In preparing this paper, the author used ChatGPT-5.6 Sol in the
following ways:
\begin{itemize}
\item Some ideas incorporated into the proofs in Sections 4 and 5 were
suggested by ChatGPT-5.6 Sol.
\item ChatGPT-5.6 Sol was used to improve the exposition and to help
identify grammatical errors and possible mathematical errors.
\end{itemize}
All AI-generated suggestions and text incorporated into the paper were
carefully checked, independently verified, and revised by the author.
The author takes full responsibility for the correctness of all mathematical
statements and proofs and for the final text of this paper.

\subsection*{\normalsize Mao NAGAMINE}
\noindent
Department of Mathematics \\
Graduate School of Science\\
Hokkaido University\\
Sapporo 060-0810, JAPAN\\
e-mail: nagamine.mao.d9@elms.hokudai.ac.jp


\begin{thebibliography}{99}


\bibitem{BjornerWachs}
A.~Bj\"orner and M.~L.~Wachs,
\emph{Shellable nonpure complexes and posets. II},
Trans. Amer. Math. Soc. \textbf{349} (1997), no.~10, 3945--3975.

\bibitem{GKZbook}
I.~M.~Gel'fand, M.~M.~Kapranov, and A.~V.~Zelevinsky,
\emph{Discriminants, Resultants, and Multidimensional Determinants},
Birkh\"auser, Boston, 1994.

\bibitem{Hotta}
R.~Hotta,
\emph{Equivariant $D$-modules},
in: Proceedings of the ICPAM Spring School in Wuhan, 1991,
available as arXiv:math/9805021.

\bibitem{MillerSturmfels}
E.~Miller and B.~Sturmfels,
\emph{Combinatorial Commutative Algebra},
Graduate Texts in Mathematics, vol.~227, Springer, New York, 2005.

\bibitem{NagamineCore}
M.~Nagamine,
\emph{$A$-hypergeometric Series with Parameters in the Core},
Math. J. Okayama Univ., to appear.

\bibitem{OkuyamaSaitoII}
G.~Okuyama and M.~Saito,
\emph{Logarithmic $A$-hypergeometric series II},
Beitr. Algebra Geom. \textbf{64} (2023), no.~4, 1057--1086.

\bibitem{Saito2002}
M.~Saito,
\emph{Logarithm-free $A$-hypergeometric series},
Duke Math. J. \textbf{115} (2002), no.~1, 53--73.

\bibitem{SaitoLog}
M.~Saito,
\emph{Logarithmic $A$-hypergeometric series},
Int. J. Math. \textbf{31} (2020), no.~13, 2050110.

\bibitem{SST}
M.~Saito, B.~Sturmfels, and N.~Takayama,
\emph{Gr\"obner Deformations of Hypergeometric Differential Equations},
Algorithms and Computation in Mathematics, vol.~6,
Springer, Berlin, 2000.

\bibitem{SchulzeWalther}
M.~Schulze and U.~Walther,
\emph{Irregularity of hypergeometric systems via slopes along coordinate subspaces},
Duke Math. J. \textbf{142} (2008), no.~3, 465--509.


\end{thebibliography}
\end{document}